\documentclass[a4paper,12pt]{amsart}

\usepackage{amsmath, amssymb, amsthm}
\usepackage{geometry}
\usepackage{hyperref}
\usepackage{xcolor}
\def\N{\mathbb N}

\def\NN{\mathbb N}
\def\CC{\mathbb C}
\def\CC{\mathbb C}
\def\DD{\mathbb D}
\def\R{\mathbb R}
\def\RR{\mathbb R}
\def\ZZ{\mathbb Z}

\newtheorem{thm}{Theorem}[section]
\newtheorem{corollary}[thm]{Corollary}
\newtheorem{proposition}[thm]{Proposition}
\newtheorem{theorem}[thm]{Theorem}
\newtheorem{lemma}[thm]{Lemma}

\theoremstyle{definition}
\newtheorem{definition}[thm]{Definition}
\newtheorem{example}[thm]{Example}
\theoremstyle{remark}

\title{ Entire groups generated by fractional powers of operators}
\author{Rodrigo Barber\'an and Pedro J. Miana}

\address{Departamento de
Matem\'aticas, Instituto Universitario de Matem\'aticas y Aplicaciones, Universidad de Zaragoza,  50009 Zaragoza, Spain.}
 \email{rbarberan@unizar.es, pjmiana@unizar.es}
\date{}

\thanks{ This first author has been partially supported
the project PTA2024-024763-I, of MICIU/AEI/10.13039/501100011033 and FSE+. The second author is supported by PID2022-137294NB-I00, DGI-FEDER, of the MCEI and Project E48-20R, Gobierno de Arag\'on, Spain}



\begin{document}
\maketitle

\begin{abstract} Let $T$ be a power-bounded operator on a Banach space $X$. We treat the sequence of polynomials $(p_{n;a})_{n\ge 0}$ such that the entire group generated by the fractional power operator $-(I-T)^a$ is given by
$$
e^{-t(I-T)^a}=e^{-t}\sum_{n=0}^\infty p_{n;a}(t)T^n, \qquad t\in \CC, \quad \Re a>0.
$$

We provide a self-contained introduction to the polynomial family $(p_{n;a})_{n\ge 0}$, for $a\in \CC$, whose coefficients are determined by means of a suitable recurrence relation. The sequence $(n!p_{n;a})_{n\ge 0}$ forms a family of Sheffer polynomials. For $\Re a>0$ and $t\in \CC$, the sequence $(p_{n;a}(t))_{n\ge 0}$ belongs to the Lebesgue sequence space $\ell^1$ of absolutely summable sequences. Moreover, these polynomials are closely related to the L\'evy density functions $(f_{t,\alpha})_{t>0}$ defined for $0<\alpha<1$.

Finally, we discuss several particular cases corresponding to specific values of $a\in \CC$, as well as applications to fractional powers in the Banach algebra $\ell^1$, multiplication operators, and Cesàro means.
\end{abstract}

\section*{Introduction}

The so-called stable L\'{e}vy distribution, for $0<\alpha<1$, is defined by
\begin{equation}\label{levyfunct}
f_{t,\alpha}(r) :=  \frac{1}{2\pi i} \int_{\sigma - i \infty}^{ \sigma + i \infty} e^{ z r - t z^{\alpha}} \, dz, \qquad \sigma >0, \quad t>0, \,\, r \geq 0,
\end{equation}
where the branch of $z^{\alpha}$ is chosen so that ${\Re}(z^{\alpha})>0$ whenever ${\Re}(z)>0$. It is well known that, in general, no explicit representation is available for the L\'evy distribution, except in the particular case $\alpha=\frac{1}{2}$, where one has
\begin{equation*}
f_{t,\frac{1}{2}}(r)=\frac{t}{\sqrt{4\pi r^3}}e^{\frac{-t^2}{4r}}, \qquad  t,r>0.
\end{equation*}
These functions were introduced by S. Bochner  in the study of certain stochastic processes. They constitute the density functions associated with stable L\'evy processes and are closely related to fractional Brownian motion (\cite{Bochner1}).

In \cite[Section IX.11]{Yo80}, K. Yosida employed these functions systematically in the study of $C_0$-semigroups of operators generated by fractional powers of infinitesimal generators of uniformly bounded $C_0$-semigroups of linear operators on a Banach space $X$ . More precisely, let $(T(s))_{s>0}$ be a uniformly bounded $C_0$-semigroup generated by $A$. Then the fractional power $-(-A)^\alpha$ generates a uniformly bounded $C_0$-semigroup $(e^{-t(-A)^\alpha})_{t>0}$ given by
\begin{equation}\label{yosid}
e^{-t(-A)^\alpha}(x)=\int_0^\infty f_{t,\alpha}(s) \, T(s)(x)\, ds, \qquad x\in X.
\end{equation}

We now consider the Lebesgue Banach space $(L^1(\R^+),\Vert \cdot \Vert_1)$ of absolutely integrable functions on $(0,\infty)$, endowed with the norm
$$
\Vert f \Vert_1:=\int_0^\infty \vert f(\lambda)\vert \, d\lambda<\infty.
$$
It holds that $ f_{t,\alpha}(\lambda) \geq 0$ for $\lambda >0$, $\Vert f_{t,\alpha} \Vert_1=1$, and
$$
f_{t+s,\alpha} (\lambda)= \int_0^{\lambda} f_{t,\alpha}(\lambda -\mu)f_{s,\alpha}(\mu)\, d\mu, \quad \lambda >0, \quad t,s>0.
$$
See, for instance, \cite{Yo80}.

Let $\ell^p$ denote the Lebesgue sequence space consisting of complex sequences $f=(f(n))_{n\in\N_0}$ such that
$$
\Vert f\Vert_p:=\left(\sum_{n=0}^{\infty}\lvert f(n)\rvert^p\right)^{1/p}<\infty, \qquad 1\le p<\infty.
$$
The space $\ell^\infty$ consists of bounded complex sequences endowed with the supremum norm, and $c_{0,0}$ denotes the set of sequences with finite support. It is well known that $\ell^1 \hookrightarrow \ell^p \hookrightarrow \ell^\infty$. The convolution product $\ast$ of two sequences $f=(f(n))_{n\in\N_0}$ and $g=(g(n))_{n\in\N_0}$ is defined by
$$
f\ast g(n)=\sum_{j=0}^n f(n-j) g(j), \qquad n\in \N_0.
$$
With this operation, $(\ell^1, \ast)$ is a Banach algebra, and $(\ell^p, \ast)$ is a Banach $\ell^1$-module for $1<p\le \infty$. We write $f^{1\ast}= f$ and $f^{n\ast}:=f\ast f^{(n-1)\ast}$ for $n\ge 2$.

The main motivation of this work is to identify sequences, specifically polynomial sequences $(p_{(\cdot);a}(t))_{t\in \CC}\subset \ell^1$, that play a role analogous to that of $(f_{t,\alpha})_{t>0}$ in $L^1(\R^+)$ (see Corollary \ref{levy}). As is often the case, the discrete setting (involving polynomial sequences) exhibits a richer and more intricate structure than its continuous counterpart. Applications to power-bounded operators on Banach spaces are also provided (see Section \ref{coboundaries}).

The paper is self-contained and includes the necessary preliminary material. Its structure is as follows.

In the first section, we consider the following recurrence relation: for $a\in \CC$, define $b_{1,n;a}= (a)_n$;   $b_{n, n;a}=a^n$, and
\begin{equation}\label{recur}
b_{j,n+1;a}:= (a j-n)b_{j,n;a}+ab_{j-1,n;a},  \qquad 2\le j\le n,
\end{equation}
(Definition \ref{intro}). Here, for $z\in \CC$ and $n\in \N$,  we denote by $(z)_n$ and $z^{(n)}$ the falling and rising factorials, respectively:
\begin{eqnarray*}
(z)_n&=&z(z-1)\cdots(z-n+1),\\
z^{(n)}&=&z(z+1)\cdots(z+n-1),
\end{eqnarray*}
with the usual conventions $(z)_0= z^{(0)}=1$ and $(0)_n=0^{(n)}=0$ for $n\in \NN$. Note that $(z)_n=(-1)^n(-z)^{(n)}$. These quantities are also commonly referred to as Pochhammer symbols.

For particular values of $a$ (e.g., $a=-1, {1\over2}, 1, 2,\dots$), the sequences $(b_{j,n;a})_{1\le j\le n,n}$ are well known and appear in \cite{sloane}, including, for instance, signed Lah and Bessel numbers. In particular, the Lah numbers $(L(n,j))_{n\ge 1, 1\le j\le n}$ are given by
$$
L(n,j):={n-1\choose j-1}{n!\over j!}, \qquad n\ge 1,\, 1\le j\le n,
$$
and count the number of ways to partition a set of $n$ elements into $j$ nonempty linearly ordered subsets. They were introduced by I. Lah in 1954 (\cite{Lah}) and satisfy
$$
L(n+1,j)=(n+j)L(n,j)+L(n,j-1), \quad 1\le j \le  n,
$$
which corresponds to the recurrence \eqref{recur} for $a=-1$ and then $b_{j,n;-1}=(-1)^nL(n,j)$.

Moreover, Lah numbers provide the connection between rising and falling factorials:
\begin{eqnarray}\label{exten}
 z^{(n)}&=&\sum_{j=0}^n L(n,j)(z)_{j},\\
 (z)_{n}&=&\sum_{j=0}^n(-1)^{n-j}L(n,j)z^{(j)}.\nonumber
\end{eqnarray}

 In Theorem \ref{main}, we establish the multiplicative relation
$$
(az)_n= \sum_{l=1}^n b_{l,n;a}(z)_l, \qquad n\ge 1, \, z\in \CC,
$$
which appears to be new.

In the second section, we consider the Cesàro numbers defined by $k^{\alpha}(0):=1$ and
$$
k^{\alpha}(n):={\alpha ^{(n)}\over n!}= (-1)^n {(-\alpha)_n\over n!}, \qquad n\in\N,
$$
and prove that $k^{\alpha}\in \ell^1$ if and only if $\Re\alpha <0$, together with an explicit computation of $\Vert k^{\alpha}\Vert_1$ (Theorem \ref{keys}).

The study of higher-order derivatives of functions of the form
$$
{d^{n}\over dz^n}\left(e^{ -z^{a}}\right), \qquad a\in \CC,
$$
has played a significant role in the historic development of orthogonal polynomials and special functions. Early contributions in this direction were made by R. Hoppe (1846), E.T. Bell (1934), and later H.W. Gould and A.T. Hopper (1961). For $a=m\in \N$, the Gould--Hopper polynomials $(H_n^{(m)})_{n\ge 0}$ are defined by
$$
H_n^{(m)}(z)= (-1)^n e^{z^m}{d^n\over dz^n}(e^{-{z^m}}), \qquad n\ge 0, \quad z\in \CC,
$$
and generalize the classical Hermite polynomials.

In the third section, we derive an explicit expression for the $n$-th derivative of $F_{t,a}(z):=e^{-t z^{a}}$. Although it seems to be known, at least it is given for $a\in \N$ in \cite{B}, we have decided to include to avoid the lack of completeness. In particular, we show that
$$
F^{(n)}_{t,a}(z)=e^{-t z^{a}}
\sum_{j=1}^{n}
(-1)^{j}b_{j,n;a}\,t^j z^{ja-n},
$$
and express the coefficients $b_{j,n;a}$ in terms of Bell polynomials (Theorem \ref{recur2}).

In the fourth section, for fixed $a\in \CC$, we consider the polynomial sequence $(p_{n;a})_{n\ge 0}$ defined by $p_{0;a}(t)=1$ and
$$
p_{n;a}(t):={(-1)^n\over n!}\sum_{l=1}^n b_{l,n;a}(-t)^l, \qquad n \ge 1,\quad t\in \CC.
$$
These polynomials arise as the Taylor coefficients of the holomorphic function $z\mapsto e^{t(1-(1-z)^a)}$ on the unit disc $\mathbb{D}$:
$$
\sum_{n=0}^\infty p_{n;a}(t)z^n=e^{t(1-(1-z)^a)}, \qquad \vert z\vert<1.
$$
Furthermore, they satisfy the convolution identity
$$
p_{n;a}(t+s)=\sum_{j=0}^n p_{j;a}(t)p_{n-j;a}(s), \qquad t,s\in \CC.
$$
In fact, $(n! p_{n;a})_{n\ge 0}$ forms a family of Sheffer polynomials, which can be studied via umbral calculus \cite{CG, Roman}.

For particular values of $a$, such as $a=-1$ or $a=m\in \NN$, we recover classical generating functions for Laguerre, Hermite, and Gould--Hopper polynomials. Moreover, we establish the connection with L\'{e}vy functions:
$$
\int_0^\infty {r^n\over n!}e^{-r}f_{t,\alpha}(r)\,dr= e^{-t}p_{n;\alpha}(t), \qquad t>0,
$$
for $0<\alpha<1$ and $n\in \NN_0$ (Corollary \ref{levy}).

In the fifth and sixth sections, we present several applications of our result in operator theory,  Banach algebras and  complex analysis.

Given a power-bounded operator $T$ on a Banach space $X$, we consider the fractional power $(I-T)^a$, $a\in \CC$. In Theorem \ref{poww2}, we prove that the entire group generated by $I-(I-T)^a$ admits the representation
$$
e^{t(I-(I-T)^a)}=\sum_{n=0}^\infty p_{n;a}(t)T^n, \qquad t\in \CC,
$$
under suitable assumptions on $T$ and $a$.

In the last section, we describe the semigroup generated by $(\delta_0-\delta_1)^a$ in $\ell^1$, study multiplication operators on Hardy spaces $H^p(\DD)$, and propose extensions of our results to $(C,a)$-bounded operators.

\section{Falling factorials and recurrence relation}

Given a fix $a\in \CC$, we present in this section a triangle array $(b_{j,n;a})_{n\ge 1, 1\le j\le n}$ such that
$b_{1,n;a}= (a)_n$, $b_{n, n;a}=a^n$ and $b_{j,n+1;a}$ is obtained by a recurrence relation (Definition \ref{intro}). For some particular values of $a$, for example $a\in\{-1,0,{1\over 2},1,2\}$, we obtain some known (or connections to) some known families of entire numbers, as Lah or Bessel numbers. A essential property of these numbers is the equality
$$
(az)_n= \sum_{l=1}^n b_{l,n;a}(z)_l, \qquad n\ge 1, \, z\in \CC,
$$
which is shown in Theorem \ref{main}. In this section we include several tables to ease the read of the paper.

\begin{definition}\label{intro} Take $a\in \CC$ and we define the recurrence relation  $b_{1,n;a}= (a)_n$, $b_{n, n;a}=a^n$ and
$$
b_{j,n+1;a}:= (a j-n)b_{j,n;a}+ab_{j-1,n;a}, \qquad 2\le j\le n.
$$
\end{definition}

In the following table we present the first values of the sequence $(b_{j,n;a})_{1\le j \le n,n}$.

\begin{table}[h]\label{Tableb}\small{
\centering
\begin{tabular}{c|ccccc}
$b_{j,n;a}$ & $n=1$ & $n=2$ & $n=3$ & $n=4$ & $n=5$ \\ \hline
$j=1$ & $a$ & $a(a-1)$ & $a(a-1)(a-2)$ & $a(a-1)(a-2)(a-3)$ & $a(a-1)(a-2)(a-3)(a-4)$\\ \hline
$j=2$ &$-$ &$a^2$ &$3a^2(a-1)$ & $a^2(7a-11)(a-1)$ &$5a^2(3a-5)(a-1)(a-2)$\\  \hline
$j=3$ &$-$ &$-$ &$a^3$ &$6a^3(a-1)$ & $5a^3(5a-7)(a-1)$ \\ \hline
$j=4$ &$-$ &$-$ &$-$ &$a^4$ &$10a^4(a-1)$ \\ \hline
$j=5$ &$-$ &$-$ &$-$ &$-$ &$a^5$\\
\end{tabular}
\vspace{2pt}
\caption{\small The first five $b_{j,n;a}$ for $1\le j,n\le 5.$}}
\end{table}

\begin{example}{\rm

Note that $b_{j,n;0}=0$ and $b_{j,n;1}=\delta_{j,n}$ where $\delta_{j,n}$ is the Kronecker delta,
$$
\delta_{j,n} = \begin{cases}
1, & \text{if } 1\le j = n, \\
0, & \text{if } 1\le j \neq n.
\end{cases}
$$}

Now we present the triangle $(b_{j,n;-1})_{1\le j\le n\le 5}$. As we comment in the Introduction,  an immediate consequence of the Lah recurrence relation is that $b_{j,n;-1}=(-1)^nL(n,j)$, where $L(n,j)$ are the Lah numbers.

\begin{table}[h]\label{Tableb-1}\small{
\centering
\begin{tabular}{c|ccccc}
$b_{j,n;-1}$ & $n=1$ & $n=2$ & $n=3$ & $n=4$ & $n=5$ \\ \hline
$j=1$ & $-1$ & $2$ & $-6$ & $24$& $-120$ \\
$j=2$ & $-$ & $1$ & $-6$ & $36$ & $-240$ \\
$j=3$ & $-$ & $-$ & $-1$ & $12$ & $-120$ \\
$j=4$ & $-$ & $-$ & $-$ & $1$ & $-20$ \\
$j=5$ & $-$ & $-$ & $-$ & $-$ & $-1$ \\
\end{tabular}}
\vspace{2pt}
\caption{\small The first five $b_{j,n;-1}$ for $1\le j,n\le 5.$}
\end{table}

\end{example}
For $a=2$, we obtain the following triangle,

\begin{table}[h]\label{Tableb2}\small{
\centering
\begin{tabular}{c|ccccc}
$b_{j,n;2}$ & $n=1$ & $n=2$ & $n=3$ & $n=4$ & $n=5$ \\ \hline
$j=1$ & $2$ & $2$ & $-$ & $-$ & $-$ \\
$j=2$ & $-$ & $4$ & $12$ & $12$ & - \\
$j=3$ & $-$ & $-$ & $8$ & $48$ & $120$ \\
$j=4$ & $-$ & $-$ & $-$ & $16$ & $160$ \\
$j=5$ & $-$ & $-$ & $-$ & $-$ & $32$\\
\end{tabular}}
\vspace{2pt}
\caption{\small The first five $b_{j,n;2}$ for $1\le j,n\le 5.$}
\end{table}

In the next proposition, we identify
$b_{j,n;2}$ for $1\le j\le n.$ We denote by $[x]$ and $\{x\}$ the integer part and decimal part of a real number $x$.

\begin{proposition}\label{partii} Take $n\in\{1,2\}$ and $1\le j\le n$. Then $\displaystyle b_{j,n;2}=\frac{2^{2j-n} n!}{(n-j)!(2j-n)!}$.
Now, take $n\ge 3$ and $1\le j\le n$. Then

$$b_{j,n;2}=  \begin{cases}
0, & \text{if } 1\le j <  n-[{n\over 2}], \\
\displaystyle \frac{2^{2j-n} n!}{(n-j)!(2j-n)!}, & \text{if } n-[{n\over 2}]\leq j \leq n.
\end{cases}
$$

\end{proposition}
\begin{proof}
For $n \le 2$ it is easy to verify that the formula holds: $b_{1,1;2}=2$, $b_{1,2;2}=2$ and $b_{2,2;2}=4$. \\

For $n\ge 3$, we define $c_{j,n}:=  \begin{cases}
0, & \text{if } 1\le j <  n-[{n\over 2}], \\
\displaystyle \frac{2^{2j-n} n!}{(n-j)!(2j-n)!}, & \text{if } n-[{n\over 2}]\leq j \leq n.
\end{cases}$\\

Note that $c_{n,n}=2^n=b_{n,n;2}$ and $c_{1,n}=(2)_n=b_{1,n;2}$ for $n\ge 1.$

We need to show that $c_{j,n}$ satisfies the recurrence relation for $2 \le j \le n$. For $x\in \RR,$ we write $x=[x]+\{x\}$ with $[x]\in \ZZ$ and $0\le \{x\}<1$. For $j < n-[\frac{n}{2}]$ both sides of the relation are $0$.

In the case $j > n-[\frac{n}{2}]$, we have
\begin{eqnarray*}
(2j-n)c_{j,n}+2c_{j-1,n}&=&\frac{2^{2j-n}n!}{(n-j)!(2j-n-1)!}+\frac{2^{2j-n-1}n!}{(n-j+1)!(2j-n-2)!}\\
&=&\frac{2^{2j-n-1}n!}{(n-j)!(2j-n-2)!}\left(\frac{2}{2j-n-1}+\frac{1}{n-j+1}\right)\cr&=&\frac{2^{2j-n-1}(n+1)!}{(n+1-j)!(2j-n-1)!}=c_{j,n+1}.
\end{eqnarray*}

In the boundary case $j=n-[\frac{n}{2}]$, we see that if $n$ is even, both sides of the relation are 0. If $n$ is odd, $n=2k-1$ for some natural $k$. Then $j=k$ and $c_{j-1,n}=0$ and
\begin{eqnarray*}
  (2j-n)c_{j,n}+2c_{j,n-1}&=&\frac{2^{2j-n}n!}{(n-j)!(2j-n-1)!}=\frac{2(2k-1)!}{(k-1)!}\\
  &=&\frac{(2k)!}{k!}=\frac{2^{2j-n-1}(n+1)!}{(n+1-j)!(2j-n-1)!}=c_{j,n+1}.
\end{eqnarray*}

So, the recurrence relation holds and we conclude that $c_{j,n}=b_{j,n;2}$.\end{proof}

Bessel numbers of the first kind $(a(n,j))$ are defined by
$$ a(n,j):= {(n+j)!\over 2^j(n-j)!j!}, \quad 0\le j\le n, \quad 0\le n.$$
see, for example, \cite[p.18]{G} and sequence A001498 in \cite{sloane}.

\begin{table}[h]\label{TableBN}\small{
\centering
\begin{tabular}{c|ccccc}
$a(n,j)$ &$n=0$ &$n=1$ & $n=2$ & $n=3$ & $n=4$ \\ \hline
$j=0$ & $1$ &$1$ & $1$ & $1$ & $1$ \\
$j=1$ & $-$ & $1$ & $3$ & $6$ & $10$ \\
$j=2$ & $-$ & $-$ & $3$ & $15$ & $45$ \\
$j=3$ & $-$ & $-$ & $-$ & $15$ & $105$ \\
$j=4$ & $-$ & $-$ & $-$ & $-$ & $105$ \\
\end{tabular}
\vspace{2pt}}
\caption{\small The first five $a(n,j)$ for $0\le j,n\le 4.$}
\end{table}

 {Bessel numbers} $a(n,j)$ form a triangular array of nonnegative integers that arise in combinatorics,
special functions and the theory of Bessel polynomials.
They count the number of ways to partition an
$n$-element set into
$j$ unordered pairs and singletons, where the pairs are distinguished. Numbers $(a(n,j))_{0\le j\le n}$ are coefficients of Bessel polynomials (exponents in increasing order) of degree $n$.

They satisfy the following recurrence relation
$$
a(n,j)=a(n-1,j)+{(n+j-1)}a(n-1,j-1)
$$

Now, we identify
$(b_{j,n;1/2})_{1\le j\le n}$.

\begin{table}[h]\label{Tableb1/2}
\centering\small{
\begin{tabular}{c|ccccc}
$b_{j,n;1/2}$ & $n=1$ & $n=2$ & $n=3$ & $n=4$ & $n=5$ \\\hline
$j=1$ & $1/2$ & $-1/4$ & $3/8$ & $-15/16$ & $105/32$ \\[1mm]
$j=2$ & $-$ & $1/4$ & $-3/8$ & $15/16$ & $-105/32$ \\[1mm]
$j=3$ & $-$ & $-$ & $1/8$ & $-3/8$ & $45/32$ \\[1mm]
$j=4$ & $-$ & $-$ & $-$ & $1/16$ & $-5/16$ \\[1mm]
$j=5$ & $-$ & $-$ & $-$ & $-$ & $1/32$\\[1mm]
\end{tabular}
\vspace{2pt}}
\caption{\small The first five $b_{j,n;1/2}$ for $1\le j,n\le 5.$}
\end{table}

\begin{proposition}\label{parti} Take  $n\ge 1$ and $1\le j\le n$. Then

$$
b_{j,n;1/2}={(-1)^{n+j}\over2^{n}}a(n-1,n-j).
$$

\end{proposition}
\begin{proof}

To show that, we define for a while $ \displaystyle{c_{j,n}:={(-1)^{n+j}\over2^{2n-j}}{(2n-j-1)!\over (j-1)!(n-j)!}.}$ Note that
\begin{eqnarray*}
b_{1,n;\frac{1}{2}}&=&(1/2)_n=\frac{(-1)^{n-1}(2n-3)!!}{2^n}=\frac{(-1)^{n+1}(2n-2)!}{2^{2n-1}(n-1)!}=c_{1,n},\cr
b_{n,n;\frac{1}{2}}&=&\frac{1}{2^n}=c_{n,n}.\end{eqnarray*}
Now we just need to check that the recurrence relation holds. Take $2\le j\le n$ and then
\begin{align*}
({j\over 2} -n)c_{j,n}+{1\over 2}c_{j-1,n}=(\frac{j}{2}-n)\frac{(-1)^{n+j}(2n-j-1)!}{2^{2n-j}(j-1)!(n-j)!}+\frac{1}{2}\cdot \frac{(-1)^{n+j-1}(2n-j)!}{2^{2n-j+1}(j-2)!(n-j+1)!}=\\
\left(\frac{1}{2(j-1)}+\frac{1}{4(n-j+1)}\right)\frac{(-1)^{n+j+1}(2n-j)!}{2^{2n-j}(j-2)!(n-j)!}=\frac{(-1)^{n+j+1}(2n+1-j)!}{2^{2n+2-j}(j-1)!(n+1-j)!}=c_{j,n+1}.
\end{align*}
We conclude that $c_{j,n}= b_{j,n;1/2}$ for $n\ge 1$ and $1\le j\le n$.
\end{proof}

The usual finite difference of a sequence, $(f(m))_{m\in \NN_0},$ is defined by $\Delta f(m) :=f(m +1)-f(m)$ for $m\in \NN_0$; we iterate $\Delta^{j+1}= \Delta^j(\Delta)$ to get
\begin{equation}\label{ite}
\Delta^jf(m)=\sum_{l=0}^j(-1)^{j-l} {j\choose l}f(m+l), \qquad m,j\in \NN_0.\end{equation}
Note that $\Delta z^n$ is a polynomial of order $n-1$, $\Delta^nz^n=n!$ and $\Delta^mz^n=0$ if $m>n.$ Since $\Delta^m$ is a linear operator, we conclude that $\Delta^m P=0$ in the case that $P$ is a polynomial of degree $n<m.$
\begin{theorem} Let $a\in \CC$, $n\ge 1$,  $1\le j\le n$ and $b_{j,n;a}$ given in Definition \ref{intro}. Then

$$
b_{j,n;a}= {(-1)^j\over j!}\sum_{l=1}^j(-1)^l  {j\choose l} {(al)_n}={1\over j!}\Delta^j((am)_n)(0).
$$
\end{theorem}
\begin{proof} Take $n\in \NN_0$. To show the second equality, we apply (\ref{ite}) for $f_n(m):=(am)_n$. Note that $\Delta((az)_1)=a= b_{1,1;a}$, $(a)_n=b_{1,n;a}$, and
$$
(-1)^j\sum_{l=1}^j(-1)^l{(al)_n\over l!(j-l)!}={1\over n!}\Delta^n((am)_n)(0)=a^n,
$$
since $(am)_n$ is a polynomial of degree $n$,   the leader coefficient of  $a^n$ is $\Delta^nz^n=n!$.

For a while, we write $c_{j,n;a}= \displaystyle{(-1)^j\sum_{l=1}^j(-1)^l{(al)_n\over l!(j-l)!}}$.  Since $(al)_{n+1}=(al)_n(al-n)$, we get that
\begin{eqnarray*}
c_{j,n+1;a}&=& (-1)^j\sum_{l=1}^j(-1)^l{(al)_{n+1}\over l!(j-l)!}=(-1)^ja\sum_{l=1}^j(-1)^l{(al)_{n}\over (l-1)!(j-l)!}-nc_{j,n;a}\\
&=&a\sum_{l=0}^{j-1}(-1)^{j-1-l}{(a(l+1))_{n}\over l!(j-1-l)!}-nc_{j,n;a}\\&=&{a\over (j-1)!}\Delta^{j-1}((a(m+1)_n)(0)-nc_{j,n;a}.
\end{eqnarray*}
Since $\Delta^{j-1}f(m+1)=\Delta^{j-1}f(m)+\Delta^{j}f(m) $, we obtain
\begin{eqnarray*}
c_{j,n+1;a}&=&{a\over (j-1)!}\Delta^{j-1}((am)_n)(0)+{a\over (j-1)!}\Delta^{j}((am)_n)(0)-nc_{j,n;a}\\
&=&ac_{j-1,n;a}+ajc_{j,n;a}-nc_{j,n;a}=(aj-n)c_{j,n;a}+ac_{j-1,n;a}.
\end{eqnarray*}
By Definition \ref{intro}, we conclude that $b_{j,n;a}= c_{j,n;a}$.
    \end{proof}

In the next theorem, we present a equality which extends the equality (\ref{exten}) valid for Lah numbers for general $a\in \CC.$

\begin{theorem}\label{main} Take $a,z\in \CC$ and $n\ge 1$. Then
$$
(az)_n= \sum_{l=1}^n b_{l,n;a}(z)_l
$$

\end{theorem}

\begin{proof}
For $n=1$, we have that $(az)_1=az=b_{1,1;a}(z)_1$ and the formula holds for $n=1$. Assume that the formula holds for some $n \ge 1$, i.e.,
\[
(az)_n= \sum_{l=1}^{n} b_{l,n;a} (z)_l.
\]
Note that we have $z(z)_{l-1}=(z)_l+(l-1)(z)_{l-1}$ for $l\ge 2$ and then
\begin{eqnarray*}
&\quad&(az)_{n+1}= (az-n)(az)_n=  ( a z-n) \sum_{l=1}^{n} b_{l,n;a} (z)_l=  \sum_{l=2}^{n+1} a b_{l-1,n;a}z (z)_{l-1}-\sum_{l=1}^{n} n b_{l,n;a} (z)_l \cr
&\quad&= a b_{n,n;a}z(z)_n+ \sum_{l=2}^{n} a b_{l-1,n;a}z (z)_{l-1}-\sum_{l=2}^{n} n b_{l,n;a} (z)_l -nb_{1,n;a} (z)_1\cr
&\quad&= a b_{n,n;a}z(z)_n+ \sum_{l=2}^{n} a b_{l-1,n;a}\left((z)_l+(l-1)(z)_{l-1}\right)-\sum_{l=2}^{n} n b_{l,n;a} (z)_l -nb_{1,n;a} (z)_1\cr
&\quad&= a b_{n,n;a}z(z)_n+ \sum_{l=2}^{n} a b_{l-1,n;a}(z)_l + \sum_{l=1}^{n-1}a b_{l,n;a} l(z)_{l}-\sum_{l=2}^{n} n b_{l,n;a} (z)_l -nb_{1,n;a} (z)_1\cr
&\quad&= a b_{n,n;a}(z-n)(z)_n+ \sum_{l=2}^{n}  (ab_{l-1,n;a} +(al-n) b_{l,n;a} )(z)_{l}+(a -n)b_{1,n;a} (z)_1\cr
&\quad&=  b_{n+1,n+1;a}(z)_{n+1}+ \sum_{l=2}^{n}  b_{l,n+1;a}(z)_{l}+b_{1,n+1;a} (z)_1= \sum_{l=1}^{n+1} b_{l,n+1;a} (z)_l.
\end{eqnarray*}
where we have used the  recurrence relation for $b_{l,n;a}$ given in Definition \ref{intro}.
\end{proof}

\noindent{\bf Comment} The umbral calculus is a symbolic method in combinatorics and algebra that treats sequences of polynomials in a way that mimics the manipulation of powers. Originally developed in the 19th century, it was made rigorous in the 20th century and, nowadays, is especially useful  in broader areas such as functional analysis, operator theory and special functions, see for example \cite{CG, Roman} and references therein.

Surely, Theorem \ref{main} (compare with \cite[Section 8]{Roman}) and some of the other results in this paper (for example Theorem \ref{main2} (i) and (ii)) may be shown following standard techniques of umbral calculus.

\section{Asymptotic estimation for Ces\`aro numbers}
The Euler Gamma function $\Gamma$ is defined by
$$
\Gamma(z):=\int_0^\infty {t^{z-1}}e^{-t}dt, \qquad \Re z>0.
$$
The extension of the Gamma function primarily involves analytic continuation to the entire complex plane, allowing definition for negative non-integers via the rising factorial $$\Gamma (z)=\Gamma (z+n)/z^{(n)}.$$
Handling poles at non-positive integers the following identity with the falling factorial holds
\[
(z)_n={\Gamma(z+1)\over \Gamma(z-n+1)}, \qquad n\ge 0.
\]
  The  following reflection formula holds for Euler Gamma function,
\begin{equation}
    \Gamma(z)\Gamma(1-z)={\pi \over \sin(\pi z)}, \qquad z\in \CC\setminus \ZZ.
\end{equation}

Although the next lemma is known, we include the proof to avoid the lack of completeness.
\begin{lemma}\label{121} Take $0<\Re \alpha<1$. Then
$$
\sum_{n=1}^\infty {\Gamma(n-\alpha)\over n!}= {\Gamma(1-\alpha)\over \alpha}.
$$
\end{lemma}

\begin{proof} Note that
$$
\sum_{n=1}^\infty {\Gamma(n-\alpha)\over n!}=\int_0^\infty t^{-\alpha-1}e^{-t}\sum_{n=1}^\infty {t^n\over n!}dt= \int_0^\infty t^{-\alpha-1}(1-e^{-t})dt={\Gamma(1-\alpha)\over \alpha},
$$
where we have applied the equality \cite[3.434(1)]{GR}
\end{proof}

	For any complex number $\alpha,$ we denote by $k^{\alpha}(0):=1$ and \begin{equation}\label{definition} k^{\alpha}(n):={\alpha ^{(n)}\over n!}=\frac{\alpha(\alpha+1)\cdots(\alpha+n-1)}{n!}= (-1)^n {(-\alpha)_n\over n!}\quad \text{for }n\in\N,\,\end{equation} the known Ces\`aro numbers which are studied deeply in \cite[Vol. I, p.77]{Zygmund} and are denoted by $A_{n}^{\alpha-1}$. Note that $k^0(0)=1$,  $k^0(n)=0$ for $n\ge 1$ and
    $$
    k^{-m}(n)=(-1)^n{m \choose n}, \qquad 0\le n\le m.
    $$

    The kernels $k^{\alpha}$ have played a key role in results about operator theory and fractional difference equations, see \cite{ALMV, Lizama}. Note that the sequence $k^\alpha$ can be written as $$k^{\alpha}(n)=\frac{\Gamma(n+\alpha)}{\Gamma(\alpha)\Gamma(n+1)}={n+\alpha-1\choose \alpha-1} ,\qquad n\in\N_{0},\,\alpha\in \CC\backslash\{0,-1,-2,\ldots\}.$$
 For $\alpha\in \CC$, the kernel $(k^{\alpha}(n))_{n\in\N_0}$ could be defined by the generating function, that is, \begin{equation}\label{generating}
\displaystyle\sum_{n=0}^{\infty}k^{\alpha}(n)z^n=\frac{1}{(1-z)^{\alpha}},\quad |z|<1.
 \end{equation}
 As usual, we define $z^{\alpha}=e^{\alpha\log z}$ for a branch of the complex logarithm on $\DD$.

 Note that these kernels satisfy the semigroup property, $k^{\alpha}*k^{\beta}=k^{\alpha+\beta}$ for $\alpha, \beta \in \CC$. As a function of $n,$ $k^\alpha$ is increasing  for $\alpha >1$, decreasing for $1>\alpha >0$ and $k^1(n)=1$ for $n\in \N$ (\cite[Theorem III.1.17]{Zygmund}). Furthermore, it is straightforward to check that  $k^\alpha(n)\le k^\beta(n)$ for $\beta \ge \alpha>0$ and $n\in \N_0$.

For $\alpha\in\{0,-1,-2,\ldots\}$, $k^{\alpha}(n)=0$ for $n> -\alpha,$ i.e., $k^\alpha\in c_{0,0}$. In addition, for real $\alpha\notin\{0,-1,-2,\ldots\},$   \begin{equation}\label{double}
 k^{\alpha}(n)=\frac{n^{\alpha-1}}{\Gamma(\alpha)}(1+O({1\over n})), \qquad n\in \N. \end{equation}
(\cite[Vol. I, p.77 (1.18)]{Zygmund}). Moreover this property holds in a more general context. Let $\alpha,z\in\CC,$ then \begin{equation}\label{jiji}
    \frac{\Gamma(z+\alpha)}{\Gamma(z)}=z^{\alpha}(1+\frac{\alpha(\alpha+1)}{2z}+O(|z|^{-2})),\quad |z|\to\infty,\end{equation}
    whenever $z\neq 0,-1,-2,\ldots$ and $z\neq -\alpha,-\alpha-1,\ldots,$ see \cite{ET}. Then $k^\alpha\in \ell^1$ if and only if $\Re \alpha <0$ or $\alpha=0$ (see (\ref{double}) and \cite[Lemma 5.3(i)]{AM2018}. In the next result we estimate the norm.

\begin{theorem}\label{keys} Take $\Re\alpha >0$ and $k^\alpha$ given in (\ref{definition}).
\begin{itemize}
    \item [(i)] $\Vert k^{-\alpha}\Vert_1=2$ for $0<\alpha<1$, $\Vert k^0\Vert_1=1$ and $$
     \Vert k^{-\alpha}\Vert_1\le2{\vert \alpha\vert\over \Re\alpha}{\Gamma(1-\Re\alpha)\over \vert\Gamma(1-\alpha)\vert}, \quad  0<\Re\alpha<1.$$
     \item [(ii)]  $\Vert k^{-m}\Vert_1= 2^m$ for $m\in \NN$.
     \item [(iii)]  $\Vert k^{-\alpha}\Vert_1\le 2^{ [\alpha]+1}$ for $\alpha >0$ and
     $$
     \Vert k^{-\alpha}\Vert_1\le 2^{[\Re \alpha]+1}{\vert \alpha-[\Re\alpha]\vert\over \Re\alpha-[\Re \alpha]}{\Gamma(1-(\Re\alpha-[\Re \alpha]))\over \vert\Gamma(1-(\alpha-[\Re\alpha])\vert}, \quad  0<\Re\alpha.
     $$
\end{itemize}
\end{theorem}

\begin{proof} (i) Take $0<\alpha<1$. Then
$$
\Vert k^{-\alpha}\Vert_1= \vert k^{-\alpha}(0)\vert+ {1\over \vert\Gamma(-\alpha)\vert}\sum_{n=1}^\infty{\Gamma(n-\alpha)\over n!}=1+ {1\over \vert\Gamma(-\alpha)\vert}{\Gamma(1-\alpha)\over \alpha}=2,
    $$
 where we have applied Lemma \ref{121}. Now, let $0<\Re \alpha <1$ and
 $$
\Vert k^{-\alpha}\Vert_1= \vert k^{-\alpha}(0)\vert+ {1\over \vert\Gamma(-\alpha)\vert}\sum_{n=1}^\infty{\vert\Gamma(n-\alpha)\vert\over n!}\le1+ {\vert \alpha\vert\over \Re\alpha}{\Gamma(1-\Re\alpha)\over \vert\Gamma(1-\alpha)\vert}\le 2{\vert \alpha\vert\over \Re\alpha}{\Gamma(1-\Re\alpha)\over \vert\Gamma(1-\alpha)\vert}.
    $$

 \noindent (ii) For $m\in \N$, it is straightforward to check
 $$
 \Vert k^{-m}\Vert_1= \sum_{n=0}^m{m\choose n}=2^m.
 $$
\noindent (iii) Take $\alpha=[\alpha]+\{\alpha\}$ with $0<\{\alpha\}<1$. Then
$$
\Vert k^{-\alpha}\Vert_1= \Vert k^{-[\alpha]}\ast k^{-\{\alpha\}}\Vert_1\le 2^{[\alpha]+1},
$$
 where we have applied (i) and (ii). Finally, we write $\alpha= [\Re \alpha]+\{\Re \alpha\}+i\Im\alpha$ and
 $$
 \Vert k^{-\alpha}\Vert_1\le \Vert k^{-[\Re \alpha]}\ast k^{-(\{\Re \alpha\}+i\Im\alpha)}\Vert_1\le 2^{[\Re \alpha]+1}
 {\vert \alpha-[\Re\alpha]\vert\over \Re\alpha-[\Re \alpha]}{\Gamma(1-(\Re\alpha-[\Re \alpha]))\over \vert\Gamma(1-(\alpha-[\Re\alpha])\vert}
 $$
 where we have used again (i) and (ii).
\end{proof}

\begin{corollary}\label{cor} Take $r>0$,  $\Re\alpha >0$ and $k^\alpha$ given in (\ref{definition}).
    Then the series
$$
\sum_{n,j=0}^{\infty}\vert k^{-\alpha j}(n)\vert{r^j\over j! }\le e^{c_\alpha r}<\infty,
$$
 where $c_\alpha: =2^{[\Re \alpha]+1}{\vert \alpha-[\Re\alpha]\vert\over \Re\alpha-[\Re \alpha]}{\Gamma(1-(\Re\alpha-[\Re \alpha]))\over \vert\Gamma(1-(\alpha-[\Re\alpha])\vert} $.
\end{corollary}

\begin{proof} Note that
$$
\sum_{n,j=0}^{\infty}\vert k^{-\alpha j}(n)\vert{r^j\over j! }= \sum_{j=0}^{\infty}{r^j\over j! }\Vert k^{(-\alpha)^{j\ast} }\Vert_1\le  \sum_{j=0}^{\infty}{r^j\over j! }\Vert k^{-\alpha }\Vert_1^j \le  e^{c_\alpha r}<\infty,
$$
where $k^{0\ast}:=k^0$ and we have  applied Theorem \ref{keys} (iii).
\end{proof}


\section{Bell polynomials and successive derivatives of $e^{-t z^{a}}$}

Let \(f\) and \(g\) be sufficiently smooth functions. The expression of $\frac{d^n}{dz^n} f(g(z))$ was proposed by F. Faà di Bruno in \cite{Faa} and it is  known in nowadays as the Fáa di Bruno formula. As other famous formulae,  there are several unknown previous approaches which one can find \cite{johnson}. In any case, the following equality yields
\begin{equation}\label{Faa}
\frac{d^n}{dz^n} f(g(z))
=
\sum_{k=1}^n
f^{(k)}(g(z))\,
B_{n,k}\!\big(g'(z), g''(z), \dots, g^{(n-k+1)}(z)\big),
\end{equation}
where
\(B_{n,k}(x_1,\dots,x_{n-k+1})\) are
the \emph{partial (or incomplete) Bell polynomials} defined by
\[
B_{n,k}(x_1,\dots,x_{n-k+1})
=
\sum
\frac{n!}{j_1!\cdots j_{n-k+1}!}
\prod_{m=1}^{n-k+1}
\left( \frac{x_m}{m!} \right)^{j_m},
\]
where the sum is over all nonnegative integers
\((j_1,\dots,j_{n-k+1})\) satisfying
\[
\sum_{m=1}^{n-k+1} j_m = k,
\qquad
\sum_{m=1}^{n-k+1} m j_m = n.
\]
The partial Bell polynomials count the number of ways to partition a set
of \(n\) elements into \(k\) blocks, where each block of size \(m\) is weighted
by \(x_m\).
These polynomials play a fundamental role in combinatorics and asymptotic
analysis and they were introduced by E. T. Bell in \cite{B}. If \(x_m = 1\) for all \(m\), then
\[
B_{n,k}(1,1,\dots,1) = S(n,k),
\]
where \(S(n,k)\) are the Stirling numbers of the second kind. For convenience, we list the first values of the partial Bell
polynomials in Table 6.

\begin{table}[h]
\centering
\small
\begin{tabular}{l|l}
$n$ &  $B_{n,k}(x_1, \dots, x_{n-k+1})$  \\ \hline
1 & $B_{1,1} = x_1$ \\
2 & $B_{2,1} = x_2, \quad B_{2,2} = x_1^2$ \\
3 & $B_{3,1} = x_3, \quad B_{3,2} = 3x_1x_2, \quad B_{3,3} = x_1^3$ \\
4 & $B_{4,1} = x_4, \quad B_{4,2} = 4x_1x_3 + 3x_2^2, \quad B_{4,3} = 6x_1^2x_2, \quad B_{4,4} = x_1^4$ \\
5 & $B_{5,1} = x_5, \quad B_{5,2} = 5x_1x_4 + 10x_2x_3, \quad B_{5,3} = 10x_1^2x_3 + 15x_1x_2^2, \quad B_{5,4}= 5x_1x_4,\quad B_{5,5} = x_1^5$ \\
\end{tabular}
\vspace{2pt}
\caption{Incomplete Bell Polynomials $B_{n,k}$ with $1\le k\le n\le 5$}
\end{table}

In the particular case, $\frac{d^n}{dz^n} e^{g(z)}$, we obtain that
\[
\frac{d^n}{dz^n} e^{g(z)}
=
e^{g(z)}\,
B_n\!\big(g'(z), g''(z), \dots, g^{(n)}(z)\big),
\]
where $(B_n)_{n\ge 0}$ are \emph{complete Bell polynomial} defined as
\[
B_n(x_1,\dots,x_n)
=
\sum_{k=1}^n B_{n,k}(x_1,\dots,x_{n-k+1}),
\qquad B_0 := 1.
\]
see Table 7.

\begin{table}[h]\small{
\centering
\begin{tabular}{c|l}
$n$ & $B_n(x_1, \dots, x_n)$ \\ \hline
1 & $x_1$ \\
2 & $x_1^2 + x_2$ \\
3 & $x_1^3 + 3x_1x_2 + x_3$ \\
4 & $x_1^4 + 6x_1^2x_2 + 4x_1x_3 + 3x_2^2 + x_4$ \\
5 & $x_1^5 + 10x_1^3x_2 + 15x_1x_2^2 + 10x_1^2x_3 + 5x_1x_4 + 10x_2x_3 + x_5$ \\
\end{tabular}
\vspace{2pt}}
\caption{\small The first Bell Polynomials $B_n$ for $1\le n\le 5 $}
\end{table}

The exponential generating function of the complete Bell polynomials is
\[
\exp\!\left(
\sum_{m=1}^\infty x_m \frac{t^m}{m!}
\right)
=
\sum_{n=0}^\infty
B_n(x_1,\dots,x_n)\frac{t^n}{n!}.
\]

We consider the function
\[
F_{t,a}(z):=e^{-t z^{a}},
\]
which is holomorphic on a simply connected domain
$\Omega\subset\mathbb{C}\setminus\{0\}$. Functions of the form $e^{-t z^{a}}$ appear naturally in complex analysis,
asymptotic analysis, and mathematical physics, for example in Laplace-type
integrals and fractional-order differential equations.
Understanding the structure of their higher-order derivatives is useful for
studying growth properties, singularities, and asymptotic expansions.

A direct computation shows that each derivative of $F$ factors as
\begin{equation}\label{derivative}
F^{(n)}_{t,a}(z)=P_{n,t,a}(z)\,e^{-t z^{a}},
\end{equation}
where $P_{n,t,a}$ satisfies the recurrence relation
\begin{equation}
\label{recurrence}
P_{n+1,t,a}(z)=\left({d\over dz}-t a z^{a-1}\right)P_{n,t,a}(z)=P_{n,t,a}'(z)-t a z^{a-1}P_{n,t,a}(z),
\qquad P_0(z)=1.
\end{equation}

The first derivatives illustrate the general pattern:
\begin{align*}
P_{1,t,a}(z)&=-ta z^{a-1},\\
P_{2,t,a}(z)&=t^2a^2z^{2a-2}-ta(a-1)z^{a-2},\\
P_{3,t,a}(z)&=-t^3a^3z^{3a-3}
+3t^2a^2(a-1)z^{2a-3}
-ta(a-1)(a-2)z^{a-3}.
\end{align*}
Fa\`a di Bruno's formula (\ref{Faa}) yields
\begin{equation}\label{faa}
F^{(n)}_{t,a}(z)
=
e^{-t z^a}
\sum_{j=1}^n
(-1)^j t^j z^{aj-n}
\,
B_{n,j}\!\big( (a)_1, (a)_2, \dots, (a)_{n-j+1} \big),
\end{equation}
and then
$$
P_{n,t,a}(z)=\sum_{j=1}^n
(-1)^j t^j z^{aj-n}
\,
B_{n,j}\!\big( (a)_1, (a)_2, \dots, (a)_{n-j+1} \big).
$$

A straightforward consequence of the  analyticity of the function $F_{t,a}$  is the following equality: given $\lambda\in \CC^+$,
\begin{equation}\label{serie}
e^{-t((z+\lambda)^a-\lambda^a)}=\sum_{n=0}^\infty{P_{n,t,a}(\lambda)\over n!}{z^n}, \qquad z\in D(\lambda, \Re(\lambda)):=\{w\in \CC\,\, :\,\, \vert w-\lambda \vert <\Re(\lambda)\},
\end{equation}
where this  power series converges absolutely to the function throughout the interior of its radius of convergence.

\begin{theorem}\label{recur2}
For every integer $n\ge1$,
\begin{equation}
\label{mainformula}
P_{n,t,a}(z)=
\sum_{j=1}^{n}
(-1)^{j}b_{j,n;a}\,t^j z^{ja-n}.
\end{equation}
Consequently,
\[
b_{j,n;a}
=
B_{n,j}\!\big( (a)_1, (a)_2, \dots, (a)_{n-j+1} \big).
\]
\end{theorem}
\begin{proof} We compute directly $
F'_{t,a}(z) = \frac{d}{dz} e^{-t z^a} = -a t z^{a-1} e^{-t z^a}.$
The formula gives
\[
\sum_{j=1}^{1} (-1)^{j} b_{j,1;a} t^j z^{aj - 1} F_{t,a}(z) = (-1)^{1} b_{1,1;a} t z^{a-1} F_{t,a}(z) = -b_{1,1;a} t z^{a-1} F_{t,a}(z),
\]
and since $b_{1,1;a} = a$, it matches the derivative.

Assume for some $n \ge 1$:
\[
F^{(n)}_{t,a}(z) = \left( \sum_{j=1}^{n} (-1)^{j} b_{j,n;a} t^j z^{aj - n} \right) F_{t,a}(z).
\]
We derive $F^{(n)}_{t,a}(z)$ to set,
\[
F^{(n+1)}_{t,a}(z) = \frac{d}{dz} F^{(n)}_{t,a}(z)
= \frac{d}{dz} \left[ \sum_{j=1}^{n} (-1)^{j} b_{j,n;a} t^j z^{aj - n} F_{t,a}(z) \right].
\]
We use the product rule and obtain that
\[
\frac{d}{dz} \big( z^{aj - n} F_{t,a}(z) \big)
= (aj - n) z^{aj - n - 1} F_{t,a}(z) - a t z^{a(j+1) - (n+1)} F_{t,a}(z);
\]
substitute back and,
\[
F^{(n+1)}_{t,a}(z) = \sum_{j=1}^{n} (-1)^{j} b_{j,n;a} t^j \left[ (aj - n) z^{aj - n - 1} - a t z^{a(j+1) - (n+1)} \right] F_{t,a}(z).
\]

We reindex the second sum \(j \mapsto j-1\) to get
\[
\sum_{j=2}^{n+1} (-1)^{j} a b_{j-1,n;a} t^j z^{aj - (n+1)} F_{t,a}(z).
\]

The coefficient of \(t^j z^{aj - (n+1)} F_{t,a}(z)\) is
\[
(-1)^{j} \big[ (aj - n) b_{j,n;a} + a b_{j-1,n;a} \big] = (-1)^{j} b_{j,n+1;a},
\]
which exactly matches the recurrence for \(b_{j,n+1;a}\) for $2\le j\le n$. For $j=1$, $$-(a-n)b_{1,n;a}= -(a-n)(a)_n=-(a)_{n+1}= -b_{1,n+1;a},
$$
and for $j=n+1$, $(-1)^{n+1}a b_{n,n;a}= (-1)^{n+1}a^{n+1}= (-1)^{n+1}b_{n+1,n+1;a}$.  By the Fa\`a di Bruno formula (\ref{faa}), we identify
$$
b_{j,n;a}
=
B_{n,j}\!\big( (a)_1, (a)_2, \dots, (a)_{n-j+1} \big), \qquad 1\le j\le n,
$$
and the proof is concluded.
\end{proof}

\noindent{\bf Examples.}  In the following items, we  identify some funtions $P_{n,t,a}$ for remarkable values of $a\in \CC.$
\begin{itemize}
\item[(i)] Note that for $a=1$, $P_{n,t,1}(z)= (-t)^n$.
\item[(ii)] Note that for $a=-1$, $b_{j,n;-1}=(-1)^nL(n,j)$,  see Table \ref{Tableb-1}, and then
$$
P_{n,t,-1}(z)={(-1)^nn!\over z^n}L_n^{(-1)}\left({t\over z}\right),\qquad t,z\in \CC,
$$
where $L_n^{(\alpha)}$ is the generalized Laguerre polynomial, in this case for $\alpha=-1$, see for example \cite[Chapter V]{MOS}.

\item[(iii)]  Note that for $a={1\over 2}$, we apply Proposition \ref{parti} to get that
$$
P_{n,t,{1\over 2}}(z)={(-1)^n\over (2z)^n}\sum_{j=1}^na(n-1,n-j)(\sqrt{z}t)^j= {(-1)^n\sqrt{z}t\over (2z)^n}\theta_{n-1}(\sqrt{z}t),
$$
where  $\theta_{n}(t):=\displaystyle{\sum_{k=0}^n a(n,n-k)t^k}$ is the reverse Bessel polynomial of order $n\ge 0$, see for example \cite[p.7]{G}.

\item[(iv)]Now consider $a\in \N$. For  $a=2$ and Proposition \ref{partii},  we obtain
\begin{equation}\label{Hermite}
P_{n,t,2}(z)=(-1)^nt^{n/2}H_n(\sqrt{t}\,z),\qquad t,z\in \CC,
\end{equation}
where \(H_n\) denotes the Hermite polynomial (\cite[Chapter V]{MOS}). In general, for $a=m\in \N$, we consider \(H_n^{(m)}\)  the Gould--Hopper polynomial of order \(m\), (see Introduction) and
then
\[
P_{n,t,m}(z)
= (-1)^n t^{n/m}
\,H_n^{(m)}\!\bigl(t^{1/m}z\bigr), \qquad t,z\in \CC.
\]
\end{itemize}

\section{Taylor expansions of $e^{t(1-(1-z)^a)}$ }

In the next definition we identify   polynomial sequences $(p_{n;a})_{n\ge 0}$ such that
$$\sum_{n=0}^\infty p_{n;a}(t)z^n=e^{t(1-(1-z)^a)}, \qquad \vert  z\vert<1, \qquad t,a\in \CC,
$$
see Theorem \ref{main2}(ii).

\begin{definition}\label{pol} Given $n\in \NN$ and $a\in \CC,$ we define the polynomial $p_{n;a}$ by $p_{0;a}(t)=1$ and
$$
p_{n;a}(t):={(-1)^n\over n!}\sum_{l=1}^n b_{l,n;a}(-t)^l, \qquad n \ge 1,\quad t\in \CC.
$$

\end{definition}

 Consider again holomorphic functions $F_{t,a}(z)= e^{-tz^a}$ and $P_{n,t,a}$ given in (\ref{derivative}). Then
\begin{eqnarray}\label{unno}
    p_{n;a}(t)&=&e^{t}{(-1)^n\over n!}F^{(n)}_{t,a}(1)= {(-1)^n\over n!}P_{n,t,a}(1),\\
    P_{n,t,a}(z)&= &{(-1)^n n!\over z^n}p_{n;a}(tz^a), \qquad z\in \CC^+,\label{twoo}
\end{eqnarray}
for $t\in \CC$ and $n\ge 0$, see Theorem \ref{recur2}.
In the Table 8,  we present the first values of $p_{n;a}$.

\begin{table}[h] \label{pn}
\centering
\small
\begin{tabular}{c|l}
$n$ & $n!p_{n;a}(t)$ \\ \hline
0 &
$1$
\\
1 &
$ a\, t$
\\
2 &
$a^2 t^2 - a(a-1)\, t$
\\
3 &
$ a^3 t^3
- 3a^2(a-1)\, t^2
+ a(a-1)(a-2)\, t$
\\
4 &
$a^4 t^4
- 6a^3(a-1)\, t^3
+ a^2(a-1)(7a-11)\, t^2
- a(a-1)(a-2)(a-3)\, t$
\\
5 &
$a^5 t^5
- 10a^4(a-1)\, t^4
+ 5a^3(a-1)(5a-7)\, t^3$ \\
&
$\quad
{- 5a^2(a-1)(a-2)(3a-5)\, t^2}
+a(a-1)(a-2)(a-3)(a-4)\, t$ \\

\end{tabular}
\vspace{2pt}
\caption{{First values of polynomials $n!p_{n;a}$}}
\end{table}

\smallskip
\noindent{\bf Examples.} We consider some particular values of $a$ to identify polynomials $p_{n;a}$. Note that $p_{n;0}=0$ for $n\ge 1$ and $p_{n;1}{(t)}={t^n\over n!}$ for $n\ge 0$. For $a=-1$ and $n\ge 1,$
$$
p_{n,-1}(t)=L_n^{(-1)}(t)={-t\over n}L_{n-1}(t), \qquad t\in \CC,
$$
where $L_n^{(\alpha)}$ are the generalized Laguerre polynomials.

For $a={1\over 2}$, $p_{0;{1\over 2}}(t)=1$ and
$
{p_{n;\frac{1}{2}}(t)=\displaystyle\frac{t}{2^n n!}\theta_{n-1}(t),}$ for $t\in \CC,$ and $n\ge 1,$
where  $\theta_{n}(t)$ is the reverse Bessel polynomial of order $n\ge 0$. We apply  (\ref{Hermite}) and (\ref{unno}) to get
\begin{eqnarray*}
p_{n;2}(t)&=&{{t}^{n\over 2}\over n!}H_n(\sqrt{t}),\qquad t\in \CC,\\
p_{n;m}(t)&=&{{t}^{n\over m}\over n!}H_n^{(m)}(t^{1\over m}),\qquad t\in \CC,
\end{eqnarray*}
where $H_n$ and $H_n^{(m)}$ are Hermite and Gould-Hopper polynomials, respectively.

\begin{proposition} Take $a\in \CC$ and polynomial $(p_{n;a})_{n\ge 0}$. Then
\begin{itemize}
\item[(i)] For $n>0$,
$$
(p_{n;a})'(t)={(-1)^{n{+1}}\over n!}\sum_{l=1}^n l b_{l,n;a}(-t)^{l-1}
$$
and $
 (p_{n;a})'(0)= -k^{-a}(n).
$
\item[(ii)]  For $n\ge 0$,
$$
(n+1)p_{n+1;a}(t)=(at+n)p_{n;a}(t){-}at(p_{n;a})'(t),\qquad t\in \CC.
$$
\end{itemize}
\end{proposition}

\begin{proof} (i) Note
$$
(p_{n;a})'(0)=-{(-1)^n\over n!}b_{1,n;a}=-{(-1)^n(a)_n\over n!}=-{(-a)^{(n)}\over n!}= -k^{-a}(n),
$$
for $n\ge 0.$\\
(ii) Take $t\in \CC,$
\begin{eqnarray*}
&\quad&(-1)^{n+1}(n+1)p_{n+1;a}(t)= {1\over n!}\sum_{l=1}^{n+1} b_{l,n+1;a}(-t)^l\\&\quad&\quad= {b_{1,n+1;a}\over n!}(-t)+{1\over n!}\sum_{l=2}^{n} b_{l,n+1;a}(-t)^l+ {b_{n+1,n+1;a}\over n!}(-t)^{n+1}\\
&\quad&\quad= {a(a)_{n}-n(a)_{n}\over n!}(-t)+{1\over n!}\sum_{l=2}^{n} ((a l-n)b_{l,n;a}+ab_{l-1,n;a})(-t)^l-at {a^{n}\over n!}(-t)^{n}\\
&\quad&\quad=  -{n(a)_{n}\over n!}(-t)-{at\over n!}\sum_{l=1}^{n} l b_{l,n;a}(-t)^{l-1}-{n\over n!}\sum_{l=2}^{n} b_{l,n;a}(-t)^{l}-{at\over n!}\sum_{l=1}^{n}
b_{l,n;a}(-t)^{l}\\
&\quad&\quad= -at(-1)^{n+1}(p_{n;a})'(t)-{n\over n!}\sum_{l=1}^{n} b_{l,n;a}(-t)^{l}-{at\over n!}\sum_{l=1}^{n}
b_{l,n;a}(-t)^{l}\\
&\quad&\quad= -(-1)^{n+1}at(p_{n;a})'(t)+(-1)^{n+1}{(n+at)} p_{n;a}(t)
\end{eqnarray*}
and we conclude the proof.
\end{proof}

\begin{theorem} \label{main2} Take $a\in \CC$ and polynomial $(p_{n;a})_{n\ge 0}$. Then
\begin{itemize}

    \item[(i)] For $n\ge 0$,
$$
e^{-t}p_{n;a}(t)= \sum_{j=0}^\infty {(-1)^{n}}{(-t)^{j}\over j!}{(aj)_n\over n!}= \sum_{j=0}^\infty k^{-aj}(n){(-t)^{j}\over j!}, \qquad t\in \CC.
$$

\item[(ii)] For $z\in \DD,$
$$
e^{-t}\sum_{n=0}^\infty p_{n;a}(t)z^n=e^{-t(1-z)^a}, \qquad t\in \CC;
$$
In particular $\displaystyle{\sum_{n=0}^\infty p_{n;a}(t)=e^t}$ for $\Re \alpha>0$ and $t\in \CC$.

\item[(iii)] For $\Re a>0$, $(p_{n;a}(t))_{n\ge 0}\in \ell^1$ and
$$
\Vert (p_{n;a}(t))_{n\ge 0}\Vert_1\le e^{(c_a+1)\vert t\vert}, \qquad t\in \CC,
$$
where $c_a $ is given in Corollary \ref{cor}.
\end{itemize}
\end{theorem}

\begin{proof} (i) By Theorem \ref{main}, we get
$$
\sum_{j=0}^\infty k^{-aj}(n){(-t)^{j}\over j!}=\sum_{j=0}^\infty{(-t)^{j}\over j!}{(-1)^n (aj)_n\over n!}={(-1)^n\over n!}\sum_{l=1}^n {b_{l,n;a}}\sum_{j=0}^\infty{(-t)^{j}\over j!}(j)_l= e^{-t}p_{n;a}(t),
$$
for $t\in \CC.$

\noindent (ii) Now take $z\in \DD$. By the equation (\ref{serie}) for $\lambda=1$, we get
\begin{eqnarray*}
e^{-t((1-z)^a-1)}=\sum_{n=0}^\infty(-1)^n{P_{n,t,a}(1)\over n!}{z^n}= \sum_{n=0}^\infty p_{n;a}(t)z^n.
\end{eqnarray*}

\noindent (iii) Take $\Re a>0$. By the item (i),
$$
\Vert (p_{n;a}(t))_{n\ge 0}\Vert_1=\sum_{n=0}^\infty \vert p_{n;a}(t)\vert\le  e^{\vert t\vert}\sum_{n,j=0}^\infty {\vert t\vert^j\over j!}\vert k^{-aj}(n)\vert\le e^{(c_a+1)\vert t\vert},
$$
where we have applied Corollary \ref{cor}.
and we conclude the proof.
\end{proof}


\noindent{\bf Remarks.} Note that formula (\ref{twoo}) and Theorem \ref{main2} (ii) are equivalent: we apply formula (\ref{serie}) to get
\begin{eqnarray*}
\sum_{n=0}^\infty{P_{n,t,a}(\lambda)\over n!}{z^n} &=&
\sum_{n=0}^\infty p_{n,a}(t\lambda^a)(-{z\over \lambda})^n=
e^{-t((z+\lambda)^a-\lambda^a)},
\end{eqnarray*}
for $z, \lambda \in \CC^+$ and $t\in \CC.$

By Theorem \ref{main2} (ii), polynomials $(n!p_{n;a})_{n\ge 0}$ are Sheffer polynomials for $a\not=0$. A sequence of polynomials
$\{S_n(t)\}_{n\ge 0}$ is called a \emph{Sheffer sequence} if its exponential generating function
has the form
\[
\sum_{n=0}^{\infty} S_n(t)\,\frac{z^n}{n!}
= A(z)\,e^{tB(z)},
\]
where $A(t)$ and $B(t)$ are formal power series satisfying
$A(0)\neq 0$, $B(0)=0$, and $B'(0)\neq 0$ see for example \cite[Introduction]{CG} and reference therein. In our case $A(z)=1$ and $B(z)=1-(1-z)^a$.

\begin{corollary} \label{levy} Let $f_{t,\alpha}$ be the  L\'{e}vy distribution defined in (\ref{levyfunct}) for $0<\alpha<1$. Then
\begin{eqnarray*}
   \int_0^\infty {r^n\over n!}e^{-r}f_{t,\alpha}(r)dr&=& e^{-t}p_{n;\alpha}(t), \qquad t>0, n\in \NN_0 \\
    \int_0^\infty I_0(2\sqrt{rs})e^{-r}f_{t,\alpha}(r)dr&=& e^{-t}\sum_{n=0}^\infty p_{n;\alpha}(t){s^n\over n!}, \qquad t>0,  \quad s\in \CC,
\end{eqnarray*}
where  $I_0$ denotes the modified Bessel function,  $\displaystyle{I_0(2z)=\sum_{n=0}^\infty {z^{2n}\over (n!)^2}}$, \cite[Chapter III]{MOS}.

In particular $p_{n;\alpha}(t)\ge 0$ for $0<\alpha<1$ and $t>0$ and
  $$
{1\over \sqrt{\pi}}\int_0^\infty {r^{n-{3\over 2}}}e^{-r(1+\frac{t^2}{4r^2})}dr=
\frac{e^{-t}}{2^{n-1} }\theta_{n-1}(t),\qquad t\in \CC,
$$
where  $\theta_{n}$ is the reverse Bessel polynomial of order $n$.
\end{corollary}

\begin{proof} Since
$$
e^{-tz^\alpha}=\int_0^\infty f_{t,\alpha}(r)e^{-zr}dr, \qquad \Re z>0,
$$
(see (\ref{levyfunct})) then we apply  (\ref{derivative}) to get
$$
e^{-tz^\alpha}P_{n,t,\alpha}(z)=F^{(n)}_{t,\alpha}(z)={(-1)^n}\int_0^\infty r^n e^{-zr}f_{t,\alpha}(r)dr.
$$
Now take $z=1$, and we use the equality (\ref{unno}) to obtain the first equality. The second equality follows directly from the first one.

Since $f_{t,\alpha}(r)\ge 0$ for $r,t>0$ and $0<\alpha<1$, we conclude $p_{n;\alpha}(t)\ge 0$ for $0<\alpha<1$ and $t>0$. Finally take $\alpha={1\over 2}$ in the first equality to obtain the last one.
\end{proof}
\section{Entire groups generated by fractional powers}\label{coboundaries}

Let $X$ be a Banach space, ${\mathcal B}(X)$ the set of linear and bounded operators on $X$, $I$  the identity operator on $X$ and $T\in {\mathcal B}(X)$  a power bounded operator, i.e.
 \begin{equation}\label{pbb}C_T:=\sup_{n\ge 0}{\Vert T^n\Vert}<\infty,
 \end{equation}
where $T^0=I.$  Let $\Psi_T: \ell^1\to {\mathcal B}(X)$  be the linear and bounded map (also called a discrete functional calculus) given by
$$
\Psi_T(f)= \sum_{n=0}^\infty f(n)T^n, \qquad f\in \ell^1.
$$
It is straightforward to check $\Vert\Psi_T(f)\Vert\le C_T\Vert f\Vert_1$ and $\Psi_T(f\ast g)=\Psi_T(f)\Psi_T(g)$ for $f, g\in \ell^1$, see for example, \cite[Section 2]{Dun}.

As the sequence $(k^{\alpha}(n))_{n\ge 0}\in \ell^1$ for $ \Re \alpha<0$ (see Section 2), we may consider  fractional powers of $(I-T)$,
$$
(I-T)^{\alpha}:=\sum_{n=0}^{\infty}k^{-\alpha}(n)T^n,  \qquad \Re \alpha >0.
$$
(see \cite[Section 6]{LR} for $\alpha>0$). This definition of fractional power coincides with other known ones. For example, since $T$ is a power bounded operator on a Banach space $X$, then $(e^{zT})_{z\in \CC}$ is an entire group such that $\Vert e^{zT}\Vert\le C_Te^{\vert z\vert}$, in particular
 $$
 \sup_{t>0}\Vert e^{-t(I-T)}\Vert<\infty.
 $$
 It is well-known that
 $$
 (I-T)^{\alpha}x={´\sin(\alpha \pi)\over \pi}\int_0^\infty\lambda^{\alpha-1}(\lambda+I-T)^{-1}(I-T)xd\lambda={1\over \Gamma(-\alpha)}\int_0^\infty {e^{-t(I-T)}x-x\over t^{\alpha+1}}dt,
 $$
for $0<\alpha<1$ and  $x\in X$, see for example \cite[Chapter IX]{Yo80}.

As consequences of Theorem \ref{keys} and   \ref{main2} (iii), we estimate the norm of the entire group $(e^{t(I-T)^\alpha})_{t\in \CC}$ generated by $(I-T)^{\alpha}$.

\begin{corollary}\label{poww} Take $\Re \alpha >0$ and $T\in {\mathcal B}(X)$ a power bounded operator. Then
\begin{itemize}
\item[(i)] The operator $(I-T)^{\alpha}$ is bounded and
$$\Vert (I-T)^{\alpha}\Vert\le C \Vert k^{-\alpha}\Vert_1\le c_\alpha C_T,
$$
where $c_\alpha$ is given in Corollary \ref{cor}.
\item[(ii)] The entire group $(e^{-t(I-T)^\alpha})_{t\in \CC}$ verifies
$$
e^{-t(I-T)^\alpha}=e^{-t} \sum_{n=0}^\infty p_{n;\alpha}(t)T^n, \quad \Re \alpha>0.
$$
and $\Vert e^{t(I-T)^\alpha}\Vert\le C_T e^{(c_\alpha+1)\vert t\vert}$ for $t\in \CC.$
\end{itemize}

\end{corollary}

\noindent {\bf Comments and Examples}. A vector $y\in X$ is called an $\alpha$-fractional coboundary for $T$ if there exists $x\in X$ such that $(I-T)^\alpha (x)=y$. This equation is known as the Poisson equation (see \cite{ALG, DL}). The operator $(I-T)^\alpha$ can thus be interpreted as a fractional coboundary operator.

 For $0<\alpha<1$ and $x\in X$, L\'{e}vy functions allow us to consider $C_0$-semigroup theory generated by $-(I-T)^\alpha$  (see Introduction and (\cite[Chapter IX]{Yo80})). In fact both approaches coincide,
 \begin{eqnarray*}
 e^{-t(I-T)^\alpha}(x)&=&\int_0^\infty f_{t,\alpha}(s) e^{-s(I-T)}(x)ds=\sum_{n=0}^\infty\int_0^\infty f_{t,\alpha}(s) e^{-s}{(sT)^n(x)\over n!}ds\cr
 &=& e^{-t}\sum_{n=0}^\infty p_{n;\alpha}(t)T^n(x),
\end{eqnarray*}
 where we have applied Corollary \ref{levy}.

 For particular values of $a$, we obtain the following explict representation. For $\alpha={1\over 2}$
$$
 e^{-t\sqrt{I-T}}=e^{-t}\sum_{n=0}^\infty \frac{t}{2^n n!}\theta_{n-1}(t)T^n, \qquad t\in \CC,
$$
where $\theta_{-1}=1$ and $\theta_{n}$ is the reverse Bessel polynomial of order $n$. For $\alpha=m\in \NN, $

\begin{eqnarray*}
e^{-t(I-T)^2}&=&e^{-t}\sum_{n=0}^\infty{{t}^{n\over 2}\over n!}H_n(\sqrt{t})T^n, \qquad t\in \CC,\\
e^{-t(I-T)^m}&=&e^{-t}\sum_{n=0}^\infty{{t}^{n\over m}\over n!}H_n^{(m)}(t^{1\over m})T^n, \qquad t\in \CC,
\end{eqnarray*}
where $H_n$ and $H_n^{(m)}$ are Hermite and Gould-Hopper polynomials, respectively.

Now we consider a proper contraction, i.e, $T\in {\mathcal B}(X)$ such that $\Vert T\Vert<1$. In the case the operator $-(I-T)^\alpha$ is  bounded for any $\alpha\in \CC$ and generated a whole group $ (e^{-t(I-T)^\alpha})_{t\in \CC}$ for any $\alpha\in \CC$.

\begin{theorem}\label{poww2}  Take $ \alpha \in \CC$ and $T\in {\mathcal B}(X)$ a proper contraction. Then
\begin{itemize}
\item[(i)] The operator $(I-T)^{\alpha}\in {\mathcal B}(X)$.

\item[(ii)] The entire group $(e^{-t(I-T)^\alpha})_{t\in \CC}$ verifies
$$
e^{-t(I-T)^\alpha}=e^{-t} \sum_{n=0}^\infty p_{n;\alpha}(t)T^n, \qquad t\in \CC.
$$
\end{itemize}

\end{theorem}

\begin{proof} Take $x\in X$. The case of $\Re \alpha>0$ is considered in Corollary \ref{poww}. Now  take  $\Re \alpha\le 0$. To show (i), note that
$$
\Vert (I-T)^{\alpha}(x)\Vert\le \Vert x\Vert\sum_{n=0}^{\infty}\vert k^{-\alpha}(n)\vert \Vert T\Vert^n \le C\Vert x\Vert\left(1+\sum_{n=1}^{\infty}{n^{-\Re\alpha-1} }\Vert T\Vert^n\right)<\infty,
$$
where we have applied (\ref{double}) and we conclude that $(I-T)^{\alpha}\in {\mathcal B}(X)$. To prove (ii), it is clear that  $e^{-t(I-T)^\alpha}\in {\mathcal B}(X)$. Moreover,  since $k^\alpha \ast k^\beta= k^{\alpha +\beta}$ for $\alpha,\beta \in \CC$, it is straightforward to check
$$
\sum_{n=0}^\infty \vert k^{-\alpha m}(n)\vert \Vert T\Vert^n\le \left(\sum_{n=0}^\infty \vert k^{-\alpha }(n)\vert \Vert T\Vert^n\right)^m:= D^m,
$$
for $m\in \NN\cup \{0\}$.Then
\begin{eqnarray*}
    \Vert e^{-t(I-T)^\alpha}(x) \Vert&\le& \sum_{m=0}^\infty\Vert(I-T)^{\alpha m}(x)\Vert {\vert t\vert ^m  \over m!}\le \Vert x\Vert \sum_{m=0}^\infty \sum_{n=0}^\infty \vert k^{-\alpha m}(n)\vert \Vert T\Vert^n  {\vert t\vert^m  \over m!}\cr
    &\le& \Vert x\Vert \sum_{m=0}^\infty   {(D\vert t\vert)^m  \over m!}=\Vert x\Vert e^{D\vert t\vert}<\infty,\cr
\end{eqnarray*}
and we conclude the proof.
\end{proof}

 \noindent {\bf Example}. In the case that $\alpha=-1$ and $T$ is a proper contraction, we consider the entire group generated by the resolvent operator $(I-T)^{-1}$. In this case
$$
e^{-t(I-T)^{-1}}=\sum_{n=0}^\infty L_n^{(-1)}(t)T^n,\qquad t\in \CC,
$$
where $L_n^{(-1)}$ are the generalized Laguerre polynomials, \cite[Example 8.7.3]{GR}.

\section{Applications, conjectures and final comments}

\subsection{Fractional powers in the Banach algebra $\ell^1$}\label{41}

 The Dirac measures $\delta_0$ and $\delta_n$  are $\delta_n(j)=0$ if $n\not=j$ and $\delta_n(n)=1$ for $n,j\in \NN_0.$ We consider the element $\omega=\delta_0-\delta_1$. Obviously $\omega \in \ell^1$, $\Vert \omega\Vert_1=2$, and
 $$
 e^{-t\omega}(n)={e^{-t}}{t^n\over n!}, \qquad n\in \NN_0, \quad t\in \CC,
 $$
see for example, \cite[Theorem 3.3]{GLM}. In this subsection, we are interested in studying the fractional power $(\delta_0-\delta_1)^\alpha$ and the entire group $e^{-t(\delta_0-\delta_1)^\alpha}$ for $\alpha, t\in \CC.$

 Since $k^\alpha\in \ell^1$ for $\Re \alpha <0$ (see details in Theorem \ref{keys}), we  define the fractional power $(\delta_0-\delta_1)^{\alpha}\in \ell^1$ via the  generating formula (\ref{generating})
\begin{equation}\label{jaja}
(\delta_0-\delta_1)^{\alpha}:=\sum_{n=0}^{\infty}k^{-\alpha}(n)\delta_n, \qquad \Re \alpha >0.
\end{equation}
Now consider the entire group $(e^{-t(\delta_0-\delta_1)^{\alpha}})_{t\in \CC}$. By Theorem \ref{main2} (i), we get
$$
e^{-t(\delta_0-\delta_1)^{\alpha}}(n)= \sum_{j=0}^\infty{(-t)^j\over j!}(\delta_0-\delta_1)^{\alpha j}(n)=
\sum_{j=0}^\infty{(-t)^j\over j!}k^{-\alpha j}(n)=e^{-t} p_{n;\alpha}(t),
$$
for $n\ge 0$, $t\in \CC $, and we conclude that
$$
e^{-t(\delta_0-\delta_1)^\alpha}=e^{-t} \sum_{n=0}^\infty p_{n;\alpha}(t)\delta_n, \qquad t\in \CC, \Re \alpha>0.
$$

In the case that $\Re \alpha\le 0$, formula \ref{jaja} holds for $n\in \NN_0$, i.e,
$$
(\delta_0-\delta_1)^{\alpha}(n):=k^{-\alpha}(n),
$$
and then $(\delta_0-\delta_1)^{j\alpha}(n):=k^{-j\alpha}(n)$  for $n,j\in \NN_0$. We  apply the estimate (\ref{jiji}) and  Theorem \ref{main2} (i) again to conclude
$$
e^{-t(\delta_0-\delta_1)^{\alpha}}(n)=
\sum_{j=0}^\infty{(-t)^j\over j!}k^{-\alpha j}(n)=e^{-t} p_{n;\alpha}(t), \quad t\in \CC, \Re \alpha\le 0, n\ge 0.
$$

\subsection{ Multiplication operators on Hardy spaces $H^p(\DD)$}\label{multi}

Given $p\ge 1$, an analytic function $f$ belongs to the space $H^p(\mathbb{D})$ if the following norm is finite:
\[ \|f\|_p^p := \sup_{0 < r < 1} \frac{1}{2\pi} \int_0^{2\pi} |f(r e^{i\theta})|^p d\theta < \infty. \]
As usual we denote by $H^\infty(\mathbb{D})$ is the set of  bounded analytic funcions defined on the unit disc $\DD$ (see, for example \cite[Chapter 17]{rudin}).

We consider the function $h(z):=1-(1-z)^\alpha$  for $z\in \DD$. It is direct to check that $h \in H^\infty(\mathbb{D}) $ if and only if $\Re \alpha \ge 0$. In this case, we define the multiplication operator $M_h:H^p(\mathbb{D}) \to H^p(\mathbb{D})$, where $M_h(f):=hf$ for $f\in H^p(\mathbb{D})$ and $\Vert M_h\Vert= \Vert h \Vert_\infty$. Then it is clear that
$e^{th}\in H^\infty(\mathbb{D})$ and
$$
e^{th}(z)=\displaystyle\sum_{n=0}^\infty p_{n;\alpha}(t)z^n, \qquad z\in \CC,
$$
for $t\in \CC$ and $\Re a\ge 0$ (Theorem \ref{main2} (ii)); in particular
$$
\Vert e^{th}\Vert_2^2=\sum_{n=0}^\infty \vert p_{n;\alpha}(t)\vert^2, \qquad t\in \CC, \qquad \Re\alpha \ge 0.
$$
Moreover, given $f\in H^p(\DD)$, the solution of Cauchy problem
$$
\begin{cases}
\displaystyle{\partial \over \partial t} u(t,z)=M_h(u(t
,z)), &  t\ge 0, \\
u(0,z)=f(z) , &
\end{cases}
$$
is given by $u(t,z)=\displaystyle\sum_{n=0}^\infty p_{n;\alpha}(t)z^n f(z)$, see similar multiplication semigroups on $L^p$ spaces in  \cite[Section I.4]{En-Na-00}

 Now we characterize  directly when the analytic function $e^{th}\in H^p(\DD)$, i.e., $E_{t,\alpha}\in H^p(\DD)$  where
$$
E_{t,\alpha}(z):= e^{-t(1-z)^\alpha}, \qquad z\in \DD,\quad  t\in \CC.
$$

 In fact we need to study  the behaviour of $E_{t,\alpha}$ in the neighborhood  of the critical point $z=1$.  We write $\alpha = \gamma + i\beta,$ ($\gamma,\beta \in \RR$); $1-z=re^{i\theta}$,  and restrict $t\in \RR$. when $z\to 1$, then $1-z\to 0$, i.e., $r\to 0^+$ and $|\theta| \leq \frac{\pi}{2}$.
Since
$$
(1-z)^\alpha = r^\gamma e^{-\beta \theta} \left[ \cos(\beta \ln r + \gamma \theta) + i \sin(\beta \ln r + \gamma \theta) \right]
$$
then
$$
|E_{t,\alpha}(z)| = \exp\left( {\Re}\left[ -t(1-z)^\alpha \right] \right) = \exp\left( -t \cdot r^\gamma e^{-\beta \theta} \cos(\beta \ln r + \gamma \theta) \right), \qquad z\in \DD.
$$
  For $\gamma > 0$, and $r \to 0^+$, then $\lim_{z \to 1} |E_{t,\alpha}(z)| = e^0 = 1,$ and  $ E_{t,\alpha}\in H^\infty(\DD)$. Now take $\gamma = 0$,
    \[ |E_{t,\alpha}(z)| \in \left[ e^{-t e^{|\beta|\pi/2}}, e^{t e^{|\beta|\pi/2}} \right], \quad r\to 0^+, \]
and we conclude $ E_{t,\alpha}\in H^\infty(\DD)$, as we have commented previously. Finally for $\gamma<0$  and $r \to 0^+$, then $r^\gamma \to \infty$ and   $\cos(\beta \ln r + \gamma \theta) < 0$, for some $r$ and $\theta.$ Fixed $t>0$, then
    \[ \limsup_{z \to 1} |E_{t,\alpha}(z)| = \infty,  \]
and $ E_{t,\alpha}\not\in H^\infty(\DD)$.

Similarly we  prove that  $ E_{t,\alpha}\in H^p(\DD)$ for $1\le p<\infty$ if and only if $\Re \alpha \ge 0.$

\subsection{Fractional powers  for $C_\alpha$-bounded operatos} Let $T$ be a bounded lineal operator on a Banach space $X$ and we denote by $\mathcal{T}:= (T^n)_{n\ge 0}$.

Take $\alpha\ge0$ and set, for $x\in X$ and $\ n\in \N_0$,
\begin{equation*}
\Delta^{-\alpha} \mathcal{T}(n)x :
= \displaystyle\sum_{j=0}^n k^{\alpha}(n-j)T^j x,\qquad M_T^{\alpha}(n)x
:=\frac{1}{k^{\alpha+1}(n)}\Delta^{-\alpha} \mathcal{T}(n)x.
\end{equation*}
The operators $\Delta^{-\alpha} \mathcal{T}(n)$ and
$M^{\alpha}_T(n)$ in $\mathcal{B}(X)$ are called the $n$-th
Ces\`{a}ro sum and Ces\`{a}ro mean of $T$ of order $\alpha$, respectively.

The operator $T$ is called Ces\`{a}ro bounded of order $\alpha$, or simply $(C,\alpha)$-bounded, if it satisfies
$\sup_{n} \|M_T^{\alpha}(n) \| < \infty$. Thus
$(C,0)$-boundedness is the same as
power-boundedness, that is, $\sup_{n}\Vert T^n\Vert<\infty$.
For $\alpha=1$ the operator $T$ is called Ces\`{a}ro mean bounded (or Ces\`{a}ro bounded).
If $T$ is $(C,\alpha)$-bounded
then it is $(C,\beta)$-bounded for every $\beta>\alpha$ but the converse does not hold true in general;
e.g. \cite[Section 4.7]{Ed}  and  \cite[Remark 2.3]{Su-Ze13}.

Now fix $T\in {\mathcal B}(X)$ a $(C,\alpha)$-bounded operator.  In \cite[Corollary 3.7]{ALMV}, a linear homomorphism $\Psi_T: \tau^{\alpha}(k^{\alpha+1})\to {\mathcal B}(X)$ is defined by
$$
\Psi_T(f)(x)=\sum_{n=0}^\infty W^\alpha f(n) \Delta^{-\alpha} \mathcal{T}(n)(x), \qquad x\in X,
$$
where $W^\alpha$ is a fractional discrete derivation, $W^1f(n)= f(n)-f(n+1)$, $W^{m+1}f= W^{m}(W^1f)$ for $m\ge 1$, and see definition for $\alpha>0$ in  \cite{ALMV, AM2018}. Note that  $\Psi_T(\delta_1)= T$ where $\delta_1(n):=\delta_{1,n}$ for $n\ge 0$ is the  Kronecker delta. The Banach algebra $ \tau^{\alpha}(k^{\alpha+1})\subset \ell^1$ and $k^\beta\in \tau^{\alpha}(k^{\alpha+1}) $ for $\Re \beta <0$ (\cite[Lemma 5.3]{AM2018}). Then it seems natural to conjecture
\begin{eqnarray*}
    (I-T)^{\alpha}&:=& \Psi_T(k^{-\alpha}), \cr
    e^{t(I-(I-T)^{\alpha })}&:=&\Psi_T(p_{(\cdot);\alpha}(t)),
\end{eqnarray*}
for $\Re \alpha>0.$ These ideas could be the starting point of a new research.

\subsection*{Acknowledgements} Authors thank Luciano Abadias, José E. Galé, Eva A. Gallardo-Gutiérrez, and  Jes\'us Oliva-Maza for  several comments, corrections and suggestions that led to an improved version of this paper.


\end{document}